\RequirePackage{fix-cm}
\documentclass[smallcondensed]{svjour3}     % onecolumn (ditto)
\smartqed  % flush right qed marks, e.g. at end of proof
\usepackage{graphicx}
\usepackage{times,a4wide,mathrsfs}
\usepackage{amsmath}
\usepackage{amsfonts}

\newcommand{\C}{\mathbb{C}}

\newcommand{\QQ}{\mathbb{Q}}
\newcommand{\NN}{\mathbb{N}}

\newcommand{\grif}{\hbox{Griff}}

\newcommand{\gr}{\hbox{Gr}}

\newcommand{\ima}{\hbox{Im}}
\newcommand{\rom}{\romannumeral}

\newtheorem{convention}{Conventions}

\newtheorem{nonumbering}{Theorem}

\newtheorem{nonumberingc}{Corollary}

 \journalname{}
\begin{document}

\title{A short note on the weak Lefschetz property for Chow groups%\thanks{Grants or other notes
}

%\titlerunning{}        % if too long for running head

\author{Robert Laterveer %etc.
}


\institute{CNRS - IRMA, Universit\'e de Strasbourg \at
              7 rue Ren\'e Descartes \\
              67084 Strasbourg cedex\\
              France\\
              \email{laterv@math.unistra.fr}           %  \\
           }

\date{}

\maketitle

\begin{abstract} 
Motivated by the Bloch--Beilinson conjectures, we formulate a certain covariant weak Lefschetz property for Chow groups. We prove this property in some special cases, using Kimura's nilpotence theorem.

\end{abstract}

\section{Introduction}
\label{intro}

Let $X$ be a smooth projective variety over $\C$ of dimension $n$. 
The Chow groups $A^iX$ (of codimension $i$ algebraic cycles modulo rational equivalence) are notoriously hard to understand. For instance, the following conjecture dating from 1974 is still completely open for $i>1$:

\begin{conjecture}[Hartshorne \cite{H}]\label{harts} If $Y\subset X$ is a smooth hyperplane section, restriction induces isomorphisms
  \[ A^iX_{\QQ}\ \stackrel{\cong}{\to}\ A^{i}Y_{\QQ}\]
  for $2i<n-1$.
  \end{conjecture}

Since this seems a very difficult problem, in this note we try and formulate a covariant weak Lefschetz property for Chow groups and hope this is easier. To emphasize that we consider the Chow groups as a homology theory, we now switch to the notation $A_iX=A^{n-i}X$.
Let $A_i^{hom}$ and $A_i^{AJ}$ denote the subgroup of homologically trivial resp. Abel--Jacobi trivial cycles.

 To fix ideas, let's now consider $A_0X$, the Chow group of $0$--cycles. 
Since $H^{2n}(X,\QQ)$ is one--dimensional, obviously
  \[ A_0Y_{\QQ}\ \to\ {A_0X_{\QQ}/ A_0^{hom}X_{\QQ}}\]
  is surjective for any point $Y$ of $X$---and in particular, for a $0$--dimensional complete intersection $Y\subset X$. The next step is that (by weak Lefschetz applied to $H^{2n-1}(X,\QQ)$)
  \[  A_0Y_{\QQ}\ \to\ {A_0X_{\QQ}/ A_0^{AJ}X_{\QQ}}\]
  is surjective, for any smooth complete intersection curve $Y\subset X$.
  Going beyond the Abel--Jacobi map, it is conjectured there is a filtration $F^\ast$ on $A_0$, of which the first two steps are $F^1=A_0^{hom}$ and $F^2=A_0^{AJ}$ (cf. \cite{J2}, \cite{Mu}, \cite{MNP}, \cite{Vo}). One can then ask:
  
  \begin{question}\label{question} Is it true that
    \[ A_0Y_{\QQ}\ \to\ {A_0X_{\QQ}/ F^{\ell+1}}\]
    is surjective, for any smooth complete intersection $Y\subset X$ of dimension $\ell$ ?
    \end{question}
    
  This question is motivated (pun intended) by the expectation that the quotient ${A_0X_{\QQ}/ F^{\ell+1}}$ is determined by the cohomology groups $H^{2n}X, H^{2n-1}X, \ldots, H^{2n-\ell}X$. Since the filtration $F^\ast$ only exists conjecturally, this question is not falsifiable. However, it is expected that $F^{\ell+1}$ vanishes exactly when $H^nX,\ldots, H^{\ell+1}X$ are supported in codimension $1$. This gives the following conjecture, in which $F^\ast$ does not appear:
  
  \begin{conjecture}\label{conjecture} Let $X$ be a smooth projective variety, and suppose $H^i(X,\QQ)=N^1 H^i(X,\QQ)$ for $i\in[\ell+1,n]$. Then
    \[ A_0Y_{\QQ}\ \to\ A_0X_{\QQ}\]
    is surjective, for any smooth complete intersection $Y\subset X$ of dimension $\ell$.
    \end{conjecture}
    
% More generally, one can extend this to other Chow groups:
 
 % \begin{conjecture}\label{conjecture} Let $X$ be a smooth projective variety, and suppose $H^i(X,\QQ)=N^c H^i(X,\QQ)$ for $i\in [\ell+c,n]$. Then
   % \[ A_iY_{\QQ}\ \to\ A_iX_{\QQ}\]
  %  is surjective for $i<c$, for any smooth complete intersection $Y\subset X$ of dimension $\ell$.
  %  \end{conjecture}    

The main result of this note provides a verification of this conjecture in some special cases. As a by--product, we also get the injectivity part of conjecture \ref{harts} in these special cases:

\begin{nonumbering}{(=Theorem \ref{main})} Suppose the Voisin standard conjecture (conjecture \ref{csv}) holds. Let $X$ be a smooth projective variety of dimension $n$, and suppose

\item{(\rom1)} Either the motive of $X$ is finite--dimensional, or $\grif^n(X\times X)_{\QQ}=0$;

\item{(\rom2)} The Lefschetz standard conjecture $B(X)$ holds;

\item{(\rom3)} There exists $r$ such that $H^i(X,\QQ)_{}= N^r H^i(X,\QQ)$ for all $i\in [n-r+1,n]$.

Then for any codimension $r$ smooth complete intersection $Y\subset X$ of class $[Y]=L^{r}\in H^{2r}(X,\QQ)$ with $L$ ample, push--forward maps
  \[A_i(Y)_{\QQ}\ \to\ A_i(X)_{\QQ}\]
  are surjective for $i< r$. Moreover, restriction maps
    \[  A^i_{AJ}(X)_{\QQ}\ \to\ A^i_{AJ}(Y)_{\QQ}\]
    are injective for $i\le r+1$.
\end{nonumbering}

In certain cases some of the hypotheses are automatically satisfied, and the statement simplifies:

\begin{nonumberingc}{(cf. Corollary \ref{nocsv})} Let $X$ be a smooth projective 3fold which is dominated by a product of curves. Suppose 
  \[H^3(X,\QQ)=N^1 H^3(X,\QQ)\ .\] 
  Then for any smooth ample hypersurface $Y\subset X$, the push--forward map
  \[ A_0(Y)_{\QQ}\ \to\ A_0(X)_{\QQ}\]
  is surjective, and
  \[A^2_{AJ}(X)_{\QQ}\ \to\ A^2_{AJ}(Y)_{\QQ}\]
  is injective.
\end{nonumberingc}

\begin{nonumberingc}{(=Corollary \ref{nocsv2})} Let $X$ be a product of smooth projective surfaces
  \[ X=S_1\times\cdots\times S_m\ ,\]
  where each $S_j$ is either a $K3$ surface of Picard number $19$ or $20$, or has $A_0^{AJ}(S_j)_{\QQ}=0$.
  Suppose at least one $S_j$ has $A_0^{AJ}(S_j)_{\QQ}=0$.
  Then for any smooth ample hypersurface $Y\subset X$, the push--forward map
       \[ A_0(Y)_{\QQ}\ \to\ A_0(X)_{\QQ}\]
  is surjective, and
  \[A^2_{AJ}(X)_{\QQ}\ \to\ A^2_{AJ}(Y)_{\QQ}\]
  is injective.
\end{nonumberingc}

It was already known that in situations like these two corollaries, $A_0X_{\QQ}$ is supported on {\sl some\/} divisor (this follows for instance from \cite[Theorem 3.32]{Vo}); thus, our only contribution is the precision that any ample hypersurface does the job. The injectivity statement, on the other hand, seems to be genuinely new: as far as we know, these are the first examples of varieties with non--trivial $A^2_{AJ}$ for which this injectivity is known to hold.\footnote{This is not strictly true: indeed, \cite[Corollary 5]{Fu} gives non--trivial examples of varieties where the injectivity part of conjecture
\ref{harts} is verified.}
The proof of the theorem is an easy exercice in using the meccano of correspondences; the only ``deep'' ingredient is Kimura's nilpotence theorem \cite{K}.  

We end this introduction with a challenge. As is well--known \cite{BS}, the hypothesis of conjecture \ref{conjecture} is verified when $A_0X_{\QQ}$ is supported in dimension $\ell$. This gives the following special case of conjecture \ref{conjecture}:

\begin{conjecture}\label{support} Let $X$ be a smooth projective variety, and suppose $A_0X_{\QQ}$ is supported on a closed subvariety of dimension $\ell$. Then
 any smooth complete intersection $Y\subset X$ of dimension $\ell$ supports $A_0X_{\QQ}$.
 \end{conjecture}
 
 This is true for $\ell\le 1$, but for $\ell>1$ I have no idea how to prove this...

\begin{convention} In this note, the word {\sl variety\/} refers to a quasi--projective algebraic variety over $\C$. A {\sl subvariety\/} will be a (possibly reducible) reduced subscheme which is equidimensional. The Chow group of $i$--dimensional cycles on $X$ is denoted $A_iX$; for $X$ smooth of dimension $n$ the notations $A_iX$ and $A^{n-i}X$ will be used interchangeably. The Griffiths group $\grif_i$ is the group of $i$--dimensional cycles that are homologically trivial modulo algebraic equivalence. In diagrams, we will sometimes write $H^jX$ or $H_jX$ to designate singular cohomology $H^j(X,\QQ)$ resp. Borel--Moore homology $H_j(X,\QQ)$.
\end{convention}

\section{Preliminary}

\begin{definition}[Coniveau filtration \cite{BO}] Let $X$ be a quasi--projective variety. The coniveau filtration on cohomology and on homology is defined as
  \[\begin{split}   N^c H^i(X,\QQ)&= \sum \ima\bigl( H^i_Y(X,\QQ)\to H^i(X,\QQ)\bigr)\ ;\\
                           N_c H_i(X,\QQ)&=\sum \ima \bigl( H_i(Z,\QQ)\to H_i(X,\QQ)\bigr)\ ,\\
                           \end{split}\]
   where $Y$ runs over codimension $\ge c$ subvarieties of $X$, and $Z$ over dimension $\le c$ subvarieties.
 \end{definition}

We recall the statement of the ``Voisin standard conjecture'':

\begin{conjecture}[Voisin standard conjecture \cite{V0}]\label{csv} Let $X$ be a smooth projective variety, and $Y\subset X$ closed with complement $U$. Then the natural sequence
  \[  N_i H_{2i}(Y,\QQ)\to N_i H_{2i}(X,\QQ)\to N_i H_{2i}(U,\QQ)\to 0\]
  is exact for any $i$.
\end{conjecture}

\begin{remark} Hodge theory gives an exact sequence
  \[    \gr^W_{-2i} H_{2i}Y\cap F^{-i}\to H_{2i}X\cap F^{-i}\to \gr^W_{-2i} H_{2i}U\cap F^{-i}\to 0\ ,\]
  where $W$ denotes Deligne's weight filtration, and $F$ the Hodge filtration on $H_\ast(-,\C)$.
  Hence if the Hodge conjecture (that is, its homology version for singular varieties \cite{J}) is true, then conjecture \ref{csv} is true. 
  
What's more, this conjecture fits in very neatly with the classical standard conjectures: Voisin shows that conjecture \ref{csv} plus the algebraicity of the K\"unneth components of the diagonal is equivalent to the Lefschetz standard conjecture \cite[Proposition 1.6]{V0}.
 \end{remark}
  
\begin{remark}\label{csvtrue} Conjecture \ref{csv} is obviously true for $i\le 1$ (this follows from the truth of Hodge conjecture for curve classes), and for $i\ge \dim Y-1$ (where it follows from the Hodge conjecture for divisors).
\end{remark} 

The main ingredient used in this note is Kimura's nilpotence theorem:

\begin{theorem}[Kimura \cite{Kim3}]\label{nilp} Let $X$ be a smooth projective variety of dimension $n$ with finite--dimensional motive. Let $\Gamma\in A^n(X\times X)_{\QQ}$ be a correspondence which is homologically trivial. Then there is $N\in\NN$ such that
     \[ \Gamma^{\circ N}=0\ \ \ \ \in A^n(X\times X)_{\QQ}\ .\]
\end{theorem}

\begin{remark}\label{examples} We refer to \cite{Kim3}, \cite{An}, \cite{MNP} for the definition of finite--dimensional motive. 
Conjecturally, any variety has finite--dimensional motive \cite{Kim3}. What mainly concerns us in the scope of this note, is that there are quite a few examples which are known to have finite--dimensional motive:
varieties dominated by products of curves \cite{Kim3}, $K3$ surfaces with Picard number $19$ or $20$ \cite{P}, any surface with vanishing geometric genus for which Bloch's conjecture has been verified \cite[Theorem 2.11]{GP}, 3folds with nef tangent bundle \cite{I}, certain 3folds of general type \cite[Section 8]{Vial}.
\end{remark}

There is also the following nilpotence result, which predates Kimura's theorem:

\begin{theorem}[Voisin \cite{V9}, Voevodsky \cite{Voe}]\label{VV} Let $X$ be a smooth projective algebraic variety of dimension $n$, and $\Gamma\in A^n(X\times X)_{\QQ}$ a correspondence which is algebraically trivial. Then there is $N\in\NN$ such that
     \[ \Gamma^{\circ N}=0\ \ \ \ \in A^n(X\times X)_{\QQ}\ .\]
    \end{theorem}

\section{Main}

We now proceed with the proof of the main result of this note:

\begin{theorem}\label{main} Suppose the Voisin standard conjecture holds. Let $X$ be a smooth projective variety of dimension $n$, and suppose

\item{(\rom1)} Either the motive of $X$ is finite--dimensional, or $\grif^n(X\times X)_{\QQ}=0$;

\item{(\rom2)} The Lefschetz standard conjecture $B(X)$ holds;

\item{(\rom3)} $H^i(X,\QQ)_{}= N^r H^i(X,\QQ)$ for all $i\in [n-r+1,n]$.

Then for any codimension $r$ smooth complete intersection $Y\subset X$ of class $[Y]=L^{r}\in H^{2r}(X,\QQ)$ with $L$ ample, push--forward maps
  \[A_i(Y)_{\QQ}\ \to\ A_i(X)_{\QQ}\]
  are surjective for $i< r$.
 % Moreover,
  %  \[ \grif_i(Y)_{\QQ}\ \to\ \grif_i(X)_{\QQ}\]
   % is surjective for $i\le\min(c,r)$. (this is clear since \grif_i=0 !!)
   Moreover, restriction maps
    \[  A^i_{AJ}(X)_{\QQ}\ \to\ A^i_{AJ}(Y)_{\QQ}\]
    are injective for $i\le r+1$.
    \end{theorem}
    
 In certain cases, some of the hypotheses can be removed:
 
\begin{corollary}\label{nocsv} Let $X$ be a smooth projective variety of dimension $n\le 3$, and suppose

\item{(\rom1)} Either the motive of $X$ is finite--dimensional, or $\grif^n(X\times X)_{\QQ}=0$;

%\item{(\rom2)} $B(X)$ holds;

\item{(\rom2)} $H^n(X,\QQ)_{}= N^1H^n(X,\QQ)$.
%N^{\llcorner{n-1\over 2}\lrcorner} H^n(X,\QQ)$.

Then for any smooth ample hypersurface $Y\subset X$, push--forward maps
  \[A_0(Y)_{\QQ}\ \to\ A_0(X)_{\QQ}\]
  are surjective, and restriction
  \[ A^2_{AJ}(X)_{\QQ}\ \to\ A^2_{AJ}(Y)_{\QQ}\]
  is injective.
 % Moreover,
  %  \[ B_1(Y)_{\QQ}\ \to\ B_1(X)_{\QQ}\]
   % is surjective. 
   
    \end{corollary}
   
   \begin{corollary}{\label{nocsv2}} Let $X$ be a product of smooth projective surfaces
  \[ X=S_1\times\cdots\times S_m\ ,\]
  where each $S_j$ is either a $K3$ surface of Picard number $19$ or $20$, or has $A_0^{AJ}(S_j)_{\QQ}=0$.
  
  \item{(\rom1)}
  Suppose at least one $S_j$ has $A_0^{AJ}(S_j)_{\QQ}=0$.
  Then for any smooth ample hypersurface $Y\subset X$, 
       \[ A_0(Y)_{\QQ}\ \to\ A_0(X)_{\QQ}\]
  is surjective, and
  \[A^2_{AJ}(X)_{\QQ}\ \to\ A^2_{AJ}(Y)_{\QQ}\]
  is injective.
  
  \item{(\rom2)} Suppose there are at least $4$ surfaces $S_j$ with $A_0^{AJ}(S_j)_{\QQ}=0$.  
  Let $Y\subset X$ be a codimension $2$ complete intersection of class $[Y]=L^2\in H^4(X,\QQ)$ with $L$ ample.
  Then 
    \[ A_i(Y)_{\QQ}\ \to\ A_i(X)_{\QQ}\]
  is surjective for $i\le 1$, and
  \[A^i_{AJ}(X)_{\QQ}\ \to\ A^i_{AJ}(Y)_{\QQ}\]
  is injective for $i\le 3$.  
\end{corollary}

  \begin{proof}(of theorem \ref{main}) 
   Let $\tau\colon Y\hookrightarrow X$ be a smooth complete intersection of class $L^{r}$ as in the statement of the theorem. 
 Let
   \[ L^j\colon H^iX(\QQ)\to H^{i+2j}(X,\QQ)\]
   denote the result of cupping with a power of $L$; we use the same notation $L^j$ for the correspondence inducing this action. Since $B(X)$ is true, for any $i<n$ there exists a correspondence $C_i\in A^{i}(X\times X)_{\QQ}$ inducing an isomorphism
     \[ (C_i)_\ast\colon  H^{2n-i}(X,\QQ)\ \stackrel{\cong}{\to}\ H^i(X,\QQ)\]
  that is inverse to $L^{n-i}$.    
   
 $B(X)$ being true, the K\"unneth components $\pi_i$ of the diagonal of $X$ are algebraic \cite{K}. Since $B(X)$ implies $B(Y)$ \cite{K}, the same holds for the 
  K\"unneth components $\pi_i^Y$ of $Y$. We now proceed to relate them:
  
\begin{lemma} For each $i\le n-r$, define
  \[  \Pi_i:=(C_i)\circ(L^{n-i-r})\circ ((\tau\times\tau)_\ast (\pi_i^Y))\ \ \in A^{n}(X\times X,\QQ)\ .\]
  Then for each $i\le n-r$, we have equality
  \[ \Pi_i=\pi_i\ \ \in H^{2n}(X\times X,\QQ)\ .\]
  \end{lemma}  
  
  \begin{proof} We consider the action on $H^j(X,\QQ)$. There is a factorization
    \[ \begin{array}[c]{ccccccc}
        H^jX&
        \xrightarrow{((\tau\times\tau)_\ast (\pi_i^Y))_\ast}& H^{j+2r}X&\xrightarrow{L^{n-i-r}}&  H^{2n-2i+j}X&\xrightarrow{(C_i)_\ast}&H^jX\\
        \downarrow&&\uparrow &&&&\\
        H^jY&\xrightarrow{(\pi_i^Y)_\ast}& H^jY &&&&\\
        \end{array}\]
    Hence, if $j\not=i$ then
      \[ (\Pi_i)_\ast H^jX=0\ ,\]
      and for $j=i$ we have
      \[ \Pi_i=(C_i)\circ(L^{n-i-r})\circ(L^r)=\hbox{id}\ \ \colon H^iX\to H^iX\ .\]
      It follows that for any variety $Z$, the action of $\Pi_i$ on $H^j(X\times Z)$ is projection on $H^iX\otimes H^{j-i}Z$; thus by Manin's identity principle, $\Pi_i$ and $\pi_i$ coincide as homological correspondences.
 \end{proof} \qed

  \begin{lemma}\label{noaction} For each $i\le n-r$, and each $j<r$, we have
    \[ (\Pi_i)_\ast A_jX_{\QQ}=0\ .\]

 % Moreover, for each $i\le n-r$
  %  \[  (\Pi_i)_\ast \grif_rX_{\QQ}=0\ .\]
    \end{lemma}
    
  \begin{proof} For any correspondence $C\in A^{n-r}(Y\times Y)_{\QQ}$, there is a factorization
  \[\begin{array}[c]{ccc}
         A_jX_{\QQ}&\xrightarrow{((\tau\times\tau)_\ast C)_\ast}& A_{j-r}X_{\QQ}\\
         \downarrow&&\uparrow\\
         A_{j-r}Y_{\QQ}&\xrightarrow{C_\ast}&A_{j-r}Y_{\QQ}\\
         \end{array}\]
   In particular, taking $C=\pi_i^Y$, we see that the action of $(\tau\times\tau)_\ast (\pi_i^Y)$ on $A_jX_{\QQ}$ factors over $A_{j-r}Y_{\QQ}$, hence is $0$ for
   $j<r$.
   
%   The statement for the Griffiths group is similar; here we use that $\grif_0Y_{\QQ}=0$.
  \end{proof} \qed

 \begin{lemma}\label{supported} Let ${}^t\Pi_i$ denote the transpose of $\Pi_i$. For each $i\le n-r$, and each $j$, we have
   \[ 
               ({}^t\Pi_i)_\ast A_jX_{\QQ}\subset \ima\bigl( A_jY_{\QQ}\to A_jX_{\QQ}\bigr)\  .\\
                              \]
 Moreover, for each $j\le r+1$, we have
    \[ ({}^t\Pi_i)_\ast A^j_{AJ}X_{\QQ}=0\ .\]
    \end{lemma} 
   
\begin{proof} It is immediate from the definition that  
  \[ {}^t\Pi_i={} ((\tau\times\tau)_\ast ({}^t \pi_i^Y))\circ {}^t(L^{n-i-r})\circ{}^tC_i\ \ \in A^n(X\times X)_{\QQ}\ .\]
  Using the same diagram as in the proof of lemma \ref{noaction}, one can find a factorization
  \[\begin{array}[c]{ccccc}
       A_jX_{\QQ}&\ \ \xrightarrow{({}^t(L^{n-i-r})\circ{}^tC_i)_\ast}\ \ \ &          A_{j+r}X_{\QQ}&\xrightarrow{{}((\tau\times\tau)_\ast ({}^t\pi_i^Y))_\ast}& A_{j}X_{\QQ}\\
        && \downarrow&&\uparrow\\
        &&  A_jY_{\QQ}&\xrightarrow{{}^t(\pi_i^Y)_\ast}&A_jY_{\QQ}\ ,\\
         \end{array}\]
       and the lemma is proven.
\end{proof} \qed

%\[\begin{array}[c]{ccccccc}
 % H^jX&\stackrel{\bigl((\tau\times\tau)_\ast(\pi_i)\bigr)_\ast}{\to}&   H^{j+2r}X&\stackrel{L^{n-i-r}}{\to}&   H^{2n-2i+j}X &\stackrel{(C_i)_\ast}{\to}& H^jX\\
 %  \downarrow&& \uparrow && \uparrow && \downarrow\\
 %  H^jY&\stackrel{(\pi_i)_\ast}{\to}& H^jY&\stackrel{L^{n-i-r}}{\to}& H^{2n-j-2r}Y&\stackrel{\bigl((\tau\times\tau)^\ast C_i\bigr)_\ast}{\to}& H^{j}Y\\
  % \end{array}\]

 By hypothesis (\rom3), we have 
     \[H^i(X,\QQ)=N^r H^i(X,\QQ)\ \ \forall n-r<i\le n\ .\]
Applying hard Lefschetz, one finds
  \[ H^i(X,\QQ)=N^r H^i(X,\QQ)\ \ \forall n-r<i<n+r\ .\]
This means that in the range $n-r<i<n+r$, the K\"unneth component $\pi_i$ is supported in codimension $r$. That is, there exists a subvariety $Z\subset X$ of codimension $r$, such that for each $n-r<i<n+r$, $\pi_i$ goes to $0$ under the restriction
   \[  H^{2n}(X\times X,\QQ)\ \to\ H^{2n}((X\times X)\setminus (Z\times Z),\QQ)\ .\]
 Using the Voisin standard conjecture (conjecture \ref{csv}), this implies the existence of an algebraic cycle $P^\prime_i\in A_n(Z\times Z)_{\QQ}$ such that (denoting by $P_i$ the push--forward of $P_i^\prime$ to $X\times X$) we have
   \[ P_i=\pi_i\ \ \in H^{2n}(X\times X,\QQ)\ \ \forall  n-r<i<n+r\ .\]
   
   \begin{lemma}\label{noaction2} For any $i\in [n-r+1,n+r-1]$, and any $j<r$, we have
   
    \[ (P_i)_\ast A_jX_{\QQ}=0\ .\]
    Moreover, for any $j\le r+1$, we have
    \[  (P_i)_\ast A^j_{AJ}X_{\QQ}=0\ .\]
    
     \end{lemma}
    
  \begin{proof} 
  Let $\psi\colon Z\to X$ denote the inclusion, so $P_i=(\psi\times\psi)_\ast (P_i^\prime)$. Similar to lemma \ref{noaction}, there is a factorization
  \[\begin{array}[c]{ccc}
         A_jX_{\QQ}&\xrightarrow{(P_i)_\ast}& A_{j}X_{\QQ}\\
         \downarrow&&\uparrow\\
         A_{j-r}Z_{\QQ}&\xrightarrow{(P_i^\prime)_\ast}&A_{j}Z_{\QQ}\ .\\
         \end{array}\]
     That is, the action of $P_i$ in the indicated range factors over groups that vanish for dimension reasons and the lemma follows.    
          \end{proof} \qed

 Putting together the various parts, we find a decomposition of the diagonal
   \[ \Delta= \sum_{i=0}^{n-r} \Pi_i + \sum_{i={n-r+1}}^{n+r-1} P_i+\sum_{i=0}^{n-r} {}^t \Pi_i\ \ \in H^{2n}(X\times X,\QQ)\ .\]
   This is an equality of cycles modulo homological equivalence.
   Now, applying Kimura's nilpotence theorem (theorem \ref{nilp}), we get that there exists $N$ such that
   \[  \Bigl( \Delta - \sum_{i=1}^{n-r} \Pi_i - \sum_{i={n-r+1}}^{n+r-1} P_i-\sum_{i=1}^{n-r} {}^t \Pi_i  \Bigr)^{\circ N}=0\ \ \in A^n(X\times X)_{\QQ}\ .\] 
          
    Developing this expression (and noting that $\Delta^{\circ N}=\Delta$), we find 
    \[ \Delta=\sum_j Q_j\ \ \in A^n(X\times X)_{\QQ}\ ,\]
    where each $Q_j$ is a composition of elements $\Pi_\ell$ and $P_{\ell^\prime}$ and ${}^t\Pi_{\ell^{\prime\prime}}$.
    Let $Q_j^0$ denote the ``tail element'' of $Q_j$, i.e. we write
    \[ Q_j= Q_j^0\circ Q_j^1\circ\cdots\circ Q_j^{N^\prime}\ \ \in A^n(X\times X)_{\QQ}\ ,\]
    with $Q_j^0\not=\Delta$ (so that $N^\prime\le N$).
    
    Let's consider the action of $Q_j$ on $A_iX_{\QQ}$, for $i< r$:
    
    If $Q_j^0$ is a $\Pi_\ell$ (for some $\ell\in[0,n-r]$), it follows from lemma \ref{noaction} that
      \[(Q_j)_\ast \bigl(A_iX_{\QQ}\bigr)=0\ .\] 
      Likewise, if $Q_j^0$ is of the form $P_\ell$ (for some $n-r<\ell<n+r$), then applying lemma \ref{noaction2}, we find again
      \[ (Q_j)_\ast \bigl(A_iX_{\QQ}\bigr)=0\ .\]   
        Finally, if $Q_j^0$ is of the form ${}^t\Pi_\ell$ (for some $\ell\in[0,n-r]$), it follows from lemma \ref{supported} that
      \[  (Q_j)_\ast \bigl(A_iX_{\QQ}\bigr)\subset  \ima\bigl( A_iY_{\QQ}\to A_iX_{\QQ}\bigr)\ .\]
      Since $\Delta$ acts as the identity, we conclude that for $i< r$, push--forward
       \[A_iY_{\QQ}\ \to\ A_iX_{\QQ}\]
       is surjective.
       
     The argument for the injectivity statement is similar: we consider the action of $\Delta=\sum_j Q_j$ on $A^i_{AJ}X_{\QQ}$ for $i\le r+1$. If $Q_j$ is such that its ``head'' $Q_j^{N^\prime}$ is of type ${}^t\Pi_\ell$ or $P_\ell$, then $Q_j$ does not act (by lemma \ref{supported} resp. lemma \ref{noaction2}). It follows that we can write
       \[  A^i_{AJ}X_{\QQ}=\Delta_\ast A^i_{AJ}X_{\QQ}=  \bigl(\sum \hbox{something}\circ (\tau\times\tau)_\ast (\hbox{something})\bigr)_\ast A^i_{AJ}X_{\QQ}\ ;\]
      the injectivity is then obvious.        

%The statement for the Griffiths group is proven similarly (here we can go one degree higher, because we have $\grif_0=0$ for any variety).
Finally, if the hypothesis in (\rom1) of the theorem is that 
  \[\grif^n(X\times X)_{\QQ}=0\ ,\]
  the proof goes as follows: the decomposition of $\Delta$ is now an equality modulo algebraic equivalence (since by hypothesis, algebraic and homological equivalence coincide on $X\times X$).
Then, instead of applying Kimura's theorem, we apply the Voisin/Voevodsky nilpotence theorem (theorem \ref{VV}). The rest of the proof is verbatim the same.
\end{proof} \qed

\begin{proof} (of corollary \ref{nocsv}) 
In case $n=2$, we know $B(X)$ holds since it holds for any surface \cite{K}. The Voisin standard conjecture is used to get that some Hodge classes in $H_4(Z\times Z,\QQ)$ are algebraic, where $\dim Z=1$; this is trivially true.

Next, the case $n=3$.
Under the hypothesis $H^3X=N^1H^3X$, $X$ is ``motivated by a surface'' in the sense of \cite{A}, so $B(X)$ is known to hold \cite{A}.
The Voisin standard conjecture is only used to get that some Hodge classes in $H_6(Z\times Z,\QQ)$ are algebraic, where $\dim Z=2$; this is OK by the Hodge conjecture for divisors (remark \ref{csvtrue}).
\end{proof} \qed

\begin{proof} (of corollary \ref{nocsv2}) As we noted in remark \ref{examples}, it follows from work of Pedrini \cite{P} and Guletski\u{\i}--Pedrini \cite{GP} that the $S_j$ have finite--dimensional motive. Hence $X$ has finite--dimensional motive. We also know $B(X)$ is true since the Lefschetz standard conjecture is true for all surfaces \cite{K}.

In case (\rom1), since there is at least one surface with $H^2(S_j)=N^1$, we obviously have 
  \[ H^{2m}(X,\QQ)=N^1 H^{2m}(X,\QQ)\ .\]
  The corollary now follows from theorem \ref{main}; note that we don't need to assume the Voisin standard conjecture, since we can find cycles $P_i^\prime$ by using the Hodge conjecture on the surfaces with vanishing geometric genus.
  
  In case (\rom2), the assumptions imply
  \[  \begin{split}
           H^{2m}(X,\QQ)&=N^2 H^{2m}(X,\QQ)\ ;\\
           H^{2m-1}(X,\QQ)&=N^2 H^{2m-1}(X,\QQ)\ ,\\
        \end{split}  \]
      and we again apply theorem \ref{main}.              
        \end{proof}

\begin{remark} The hypothesis on $\grif^n(X\times X)$ in theorem \ref{main} is mainly of theoretical interest, and not practically useful. Indeed, there are precise conjectures predicting when Griffiths groups should vanish \cite{J3}; for instance, if $X$ is a 4fold with $h^{2,0}=h^{4,0}=h^{3,0}=h^{2,1}=0$, \cite[Corollary 6.8]{J3} implies that if the Bloch--Beilinson conjectures are true then
  \[ \grif^4(X\times X)_{\QQ}=0\ .\]
  Unfortunately, no non--trivial examples seem to be known. Specifically, I am not aware of any example of a variety $X$ of dimension $n$ that satisfies
  $\grif^n(X\times X)_{\QQ}=0$, but not $A^i_{AJ}X_{\QQ}=0\forall i$.
\end{remark}

%\begin{remark} The proof of theorem \ref{main} actually gives a bit more, in that hypothesis (\rom3) can be slightly weakened. Given the ample class $L$ defining the complete intersection $Y$, define
 % \[      H^i(X)_{r-prim}:=\hbox{Ker}\Bigl( H^i(X,\QQ)\stackrel{L^{n+r-i}}{\to} H^{2n+2r-i}(X,\QQ)\Bigr)\ .\]
 % Then hypothesis (\rom3) can be replaced the hypothesis
 % \[  H^i(X)_{r-prim}\subset N^r H^i(X,\QQ)\ \ \forall i\in [n-r+1,n]\ .\]
 % \end{remark}
  
\begin{remark} In \cite{Lat}, I study a certain hard Lefschetz property for Chow groups. Using arguments similar to the present note, this hard Lefschetz property can be proven in some special cases \cite{Lat}.
\end{remark}

\begin{acknowledgements}
This note was written while preparing for the Strasbourg ``groupe de travail'' based on the monograph \cite{Vo}. I wish to thank all the participants of this groupe de travail for the very pleasant and stimulating atmosphere.
\end{acknowledgements}



\end{document}